\documentclass[11pt,leqno]{amsart}
\usepackage[letterpaper, total={6in, 8in}, left=1.25in, top=1.5in]{geometry}
\usepackage{amsmath}
\usepackage{amsfonts}
\usepackage{amssymb}
\usepackage{graphicx}
\usepackage{color}
\usepackage{hyperref}
\usepackage{xcolor}

\newtheorem{theorem}{Theorem}[section]
\newtheorem*{theorem*}{Theorem}
\newtheorem{lemma}{Lemma}[section]

\newtheorem{proposition}{Proposition}[section]

\newtheorem{definition}[theorem]{Definition}

\newtheorem{remark}[theorem]{Remark}

\newcommand{\avint}{{\mathop{\mkern2mu\overline{\mkern-2mu\int\mkern-2mu}\mkern2mu}}}

\def\l{\lambda}

\def\p{\partial}

\def\R{\mathbb{R}}

\def\div{\operatorname{div}}

\def \p {\partial}

\def\avint{\mathop{\bar{\int}}}

\numberwithin{equation}{section}

\begin{document}

\title[The Faber-Krahn inequality]{The Faber-Krahn inequality for $p$-Hermite operators}

\author{xiaomei Sun} 
\address{College of informatics, Huazhong Agricultural University, 430070, Wuhan, China}
\email{xmsunn@mail.hzau.edu.cn}

\author{Kui Wang} 
\address{School of Mathematical Sciences, Soochow University, Suzhou, 215006, China}
\email{kuiwang@suda.edu.cn}

\author{Anqiang Zhu} 
\address{School of Mathematics and Statistics, Wuhan University, Wuhan 430072, China}
\email{aqzhu.math@whu.edu.cn}



\begin{abstract}
We prove a Faber-Krahn  inequality for the first eigenvalue of the $p$-Hermite operator (the weighted $p$-Laplacian with Gaussian weight) on Lipschitz domains in $\R^n$ under  Robin boundary conditions with positive Robin parameter. The main result states that, among all domains of given Gaussian measure, the first eigenvalue is minimized by a half-space, and equality holds only for half-spaces. This extends the classical Faber-Krahn inequalities for the $p$-Laplacian \cite{BucurCV} to the $p$-Hermite operator and generalizes the linear case \cite{ChiacchioMathann}  to the full nonlinear regime $p>1$.
\end{abstract}
\maketitle

\section{Introduction}
Let $d\gamma=\phi (x)\,dx$ denote the normalized Gaussian measure on $\R^n$, with density 
\begin{equation*}
\phi (x)=(2\pi)^{-n/2}\exp\left(-|x|^2/2\right).
\end{equation*}
Let $\Omega \subset \R^n$ be a domain with Lipschitz boundary, and  let $\nu$ denote the unit outward normal to $\p \Omega$. We consider the following eigenvalue problem 
 for the \emph{$p$-Hermite operator} (weighted $p$-Laplacian with Gaussian weight) with Robin boundary conditions:
\begin{align}\label{1.2}
\begin{cases}
-\div\bigl(\phi (x)|\nabla u(x)|^{p-2}\nabla u(x)\bigr)=\lambda(\Omega)\,\phi (x)|u|^{p-2}u, &\quad x\in \Omega,\\[4pt]
|\nabla u(x)|^{p-2}\nabla u(x)\cdot\nu+\beta\,|u(x)|^{p-2}u(x)=0, & \quad x\in \partial\Omega,
\end{cases}
\end{align}
where $p>1$ and $\beta\in \R$ is the Robin parameter. The first eigenvalue $\lambda_1(\Omega)$ admits the variational characterization
\begin{align}\label{variational formula}
\lambda_1(\Omega)=\inf_{u\in W^{1,p}(\Omega,\phi )\setminus\{0\}}
\frac{\int_{\Omega}|\nabla u|^p\phi \,dx
+\beta\int_{\partial\Omega}|u|^p\phi \,dA}
{\int_{\Omega}|u|^p\phi \,dx},
\end{align}
where $dA$ is the induced  surface measure on $\p \Omega$.
The eigenvalue problem \eqref{1.2} is understood in the weak sense: $\lambda$ is an eigenvalue if there exists a nonzero function $u\in W^{1,p}(\Omega,\phi )$, called an eigenfunction, such that 
\begin{align}\label{weak form}
    \int_{\Omega}|\nabla u|^{p-2}\nabla u\cdot \nabla v \phi (x)dx+\int_{\partial \Omega}\beta |u|^{p-2}uv\phi (x)dx=\lambda\int_{\Omega}|u|^{p-2}uv\phi (x)dx,
\end{align}
for all $v\in W^{1,p}(\Omega,\phi )$.

The classical Faber-Krahn inequality asserts that, among sets of fixed Lebesgue measure, the ball minimizes the first Dirichlet eigenvalue of the Laplace operator. For Robin boundary conditions,  the corresponding isoperimetric result was first obtained in the planar case by Bossel~\cite{BosselCR} and later extended to arbitrary dimensions by Daners~\cite{DanersMathann}. For the $p$-Laplace operator with Robin conditions, Dai-Fu~\cite{Daiqiuyi} and Bucur-Daners~\cite{BucurCV} independently established the Faber-Krahn inequality, showing that the ball is still the minimizer for Lipschitz or smooth domains. For more works on Faber-Krahn inequalities for Robin eigenvalues, see, for example \cite{BFK17, BG15, ChendaguangJDE,  CLW26, CWZ26} and the references therein.
In the Gaussian space $(\mathbb{R}^n,\phi)$, the geometric picture is very different. The celebrated Gaussian isoperimetric inequality asserts that half-spaces minimize the Gaussian perimeter among sets of fixed Gaussian measure, and they are unique up to rotation. For the Hermite operator (the linear case $p=2$) with Dirichlet boundary conditions, it is known that half-spaces minimize the first eigenvalue, in accordance with the Gaussian isoperimetric optimality (\cite{BettaZAMP,Ehrhard84,Sudakov}). For Neumann conditions, the situation is more subtle: half-spaces do not maximize the first nontrivial Neumann eigenvalue, contrary to the Euclidean Szeg\"{o}-Weinberger theorem, see~\cite{ChiacchioNeumann}.

It is therefore natural to ask which set minimizes the first eigenvalue of the Hermite operator with Robin boundary conditions in the Gaussian setting. This question was answered by Chiacchio and Gavitone~\cite{ChiacchioMathann} for the linear Hermite operator ($p=2$), who proved that the minimizer is a half-space. However, the problem for the nonlinear $p$-Hermite operator ($p\neq 2$) remained open. The nonlinearity of the $p$-Laplacian introduces new difficulties: the eigenfunction lacks sufficient regularity,  and the monotonicity properties of auxiliary functions require a careful one-dimensional analysis. Moreover, the domain $\Omega$ may be unbounded, and the first eigenfunction need not belong to $L^\infty(\Omega)$, which prevents a direct adaptation of the Euclidean  arguments. As in the linear case, a detailed study of the one-dimensional problem on half-lines is essential, together with a suitable modification of the level-set representation formula for the first eigenvalue.

In the present paper, we address the  nonlinear problem and prove a Faber-Krahn inequality for the $p$-Hermite operator with Robin boundary conditions. From now on,  we assume $\beta>0$ and $p>1$. Our main result is the following. 
\begin{theorem}\label{thm1}
    Suppose $\Omega\subset \R^n$ is a Lipschitz domain belonging to the class 
 $\mathcal G$ (see Definition \ref{Def 2.1}), and let $\Omega^{\#}$ be a half space with the same Gaussian measure as $\Omega$, i.e. $\gamma (\Omega^{\#})=\gamma (\Omega)$. 
Then 
    \begin{align}\label{eigenvalue comparison}
        \lambda_{1}(\Omega^{\#})\leq \lambda_{1}(\Omega).
    \end{align}
Moreover, equality holds if and only if  $\Omega$ is a half-space, up to a set of Gaussian measure zero.
\end{theorem}


The main novelty of this work is the extension of the Faber-Krahn inequality for Robin eigenvalues to the nonlinear $p$-Hermite operator. The nonlinearity introduces several difficulties: eigenfunctions lack regularity up to the boundary, the one-dimensional analysis requires a refined phase-plane argument for the logarithmic derivative, and the level-set representation must be adapted to the weighted measure without $L^\infty$-bounds on the eigenfunction. Our proof overcomes these by employing nonlinear PDE techniques, including weak comparison principles, Hopf's lemma, and regularity theory. This result extends  the  Faber-Krahn inequalities (\cite{BucurCV, Daiqiuyi}) for the $p$-Laplacian to $p$-Hermite operator,  and generalizes  the linear result $p=2$ (\cite{ChiacchioMathann}) to  the full nonlinear regime $p>1$.

The paper is organized as follows. In Section \ref{sect2}, we collect preliminary results on weighted Sobolev spaces and give some properties of  the  eigenvalue problem of $p$-Hermite operator.  In Section \ref{sect3}, we study the half-line problem and establish the monotonicity of the logarithmic derivative. In Section \ref{sect4}, we  derive a level-set representation formula for the first eigenvalue. In Section \ref{sect5}, we prove the main theorem.

\section{Preliminaries}\label{sect2}
In this section, we collect some definitions and preliminary results that will be used throughout the paper.
\subsection{Weighted Sobolev spaces and admissible domains}
The weighted Lebesgue space $L^{p}(\Omega, \phi )$ is defined as the set  of all measurable functions on $\Omega$ satisfying 
$$
\|u\|_{L^p}^p(\Omega, \phi ):=\int_{\Omega}|u|^p\phi (x)dx<\infty.$$
 For $1\leq p<\infty$, the weighted Sobolev space $W^{1,p}(\Omega,\phi )$ is defined by
$$W^{1,p}(\Omega,\phi ):=\{u\in W^{1,1}_{\operatorname{loc}}(\Omega): (u,|\nabla u|)\in L^p(\Omega, \phi )\times  L^p(\Omega, \phi )\},$$
which is a Banach space equipped with the norm  
$$\|u\|_{W^{1,p}(\Omega, \phi )}=\|u\|_{L^p(\Omega, \phi )}+\|\nabla u\|_{L^p(\Omega, \phi )}.$$
\begin{definition}\label{Def 2.1}
A Lipschitz domain $\Omega \subset \mathbb{R}^n$ is said to belong to the class $\mathcal{G}$ if $0<\int_\Omega \phi \, dx<1$ and the following conditions hold:
\begin{enumerate}
    \item[(i)] The embedding $W^{1,p}(\Omega, \phi ) \hookrightarrow L^p(\Omega, \phi )$ is compact.
    \item[(ii)] The trace operator
    \begin{align*}
    T : u \in W^{1,p}(\Omega, \phi ) \mapsto u|_{\partial\Omega} \in L^p(\partial\Omega, \phi )
    \end{align*}
    is well-defined.
    \item[(iii)] The trace operator $T$ is compact from $W^{1,p}(\Omega, \phi )$ onto $L^p(\partial\Omega, \phi )$.
\end{enumerate}
\end{definition}

\begin{remark} The set $\mathcal G$ contains all  bounded Lipschitz domains  in $\R^n$. Moreover,  it follows from \cite[Condition 2.1]{Feo13} that $\mathcal G$ also includes sufficiently regular unbounded domains; in particular, all half-spaces belong to $\mathcal G$.
\end{remark}
\subsection{Properties of the first eigenvalue}
For the first eigenvalue $\l_1(\Omega)$ of \eqref{1.2}, we have the following properties.
\begin{proposition}\label{Theorem1}
    Let $\Omega\in \mathcal{G}$. Then there exists a function $u\in C^{1}(\Omega)$, called the first eigenfunction, which solves \eqref{1.2} and attains the infimum in \eqref{variational formula}. The corresponding value $\lambda_1(\Omega)$ is the first eigenvalue of \eqref{1.2}. Moreover,  the first eigenfunction is positive in $\Omega$ and $\lambda_1(\Omega)$ is simple.
\end{proposition}

\begin{proof}
    The proof follows standard arguments. 
Define the energy functional
\begin{align*}
    \Phi (u):=\int_{\Omega}|\nabla u|^{p}\phi dx+\beta\int_{\partial \Omega}|u|^{p}\phi dA.
\end{align*}
The functional $\Phi$ is  convex and hence weakly lower semicontinuous (see \cite{Evans}).
By the definition of $\lambda_1(\Omega)$ and the nonnegativity of $\Phi$, there exists a minimal sequence $\{u_k\}_{k=1}^{\infty}\subset  W^{1,p}(\Omega,\phi )\setminus\{0\}$ such that 
\begin{align*}    
\|u_k\|_{L^p(\Omega,\phi )}=1,
\end{align*} 
and 
\begin{align*}
    \lim\limits_{k\rightarrow \infty}\left(\int_{\Omega}|\nabla u_{k}|^p\phi \,dx+\beta\int_{\partial\Omega}|u_{k}|^p\phi \,dA\right)=\lambda_1(\Omega).
\end{align*} 
By the compact embedding property of $\Omega\in \mathcal{G}$, there exists $u\in W^{1,p}(\Omega,\phi )$ such that, up to a subsequence,  
\begin{align*}
u_{j_{k}} \rightharpoonup u \text{ weakly in } W^{1,p}(\Omega,\phi ),
\end{align*}
\begin{align*}
u_{j_{k}} \to u \text{ strongly in } L^p(\Omega,\phi ),
\end{align*}
and 
\begin{align*}
u_{j_{k}}\rightarrow u \text{ strongly in } L^p(\partial\Omega,\phi ).
\end{align*}
Using the weak lower semicontinuity of $\Phi (u)$, we obtain
\begin{align*}
\Phi (u) \leq \liminf_{j_k\to+\infty} \Phi (u_{j_k}) = \lambda_1(\Omega).
\end{align*}
On the other hand, since $u\in W^{1, p}(\Omega, \phi)$, the definition of $\l_1(\Omega)$ gives $\lambda_1(\Omega) \leq \Phi (u)$. Hence
\begin{align*}
\lambda_1(\Omega) = \Phi (u).
\end{align*}
Now set $\psi = |u|$. It is straightforward to check that  $\lambda_1(\Omega) = \Phi (\psi)$. Moreover, by the strong maximum principle (see Lemma \ref{strong maximum principle} below), $\psi$ is strictly positive  in $\Omega$. Thus, the infimum in \eqref{variational formula} is attained at a positive function $u\in W^{1,p}(\Omega,\phi )$, which is a weak solution of \eqref{1.2}.

It remains to prove the simplicity of $\lambda_{1}(\Omega)$. The argument follows the lines of \cite{Kawohl}.
Let $\psi_{i}(i=1,2)$ be two positive first eigenfunctions normalized by
$$\|\psi_{i}\|_{L^p(\Omega,\phi )}=1.$$ 
Both are minimizers of the Rayleigh quotient \eqref{variational formula}. Define 
$$w:=(\psi_{1}^{p}+\psi_{2}^{p})^{\frac{1}{p}}.$$ Then
    \begin{align*}
        \int_{\Omega}w^{p}\phi \, dx=\int_{\Omega}\psi_{1}^{p}\phi \, dx+\int_{\Omega}\psi_{2}^{p}\phi \, dx=2.
    \end{align*}
Similarly,
    \begin{align*}
        \int_{\partial \Omega}w^{p}\phi \, dA=\int_{\partial \Omega}\psi_{1}^{p}\phi \, dA+\int_{\partial \Omega}\psi_{2}^{p}\phi \, dA.
    \end{align*}
Since $$\frac{\nabla w}{w}=\frac{\left(\psi_1^p\nabla \log \psi_1+\psi_2^p\nabla \log \psi_2\right)}{\psi_1^p+\psi_2^p},$$ 
which is a convex combination of $\nabla \log \psi_i$($i=1,2$), Jensen's inequality for convex functions yields
\begin{align*}
|\nabla w|^p\leq w^p\left(\frac{\psi_1^p|\nabla \log \psi_1|^p}{\psi_1^p+\psi_2^p}+\frac{\psi_2^p|\nabla \log \psi_2|^p}{\psi_1^p+\psi_2^p}\right)=|\nabla \psi_1|^p+|\nabla \psi_2|^p.
\end{align*}
This inequality is strict whenever $\nabla \log \psi_1\neq \nabla \log \psi_2$.
Consequently, 
    \begin{align}\label{Dirichlet term convexity}
        \int_{\Omega}|\nabla w|^{p}\phi \, dx\leq \int_{\Omega}|\nabla \psi_{1}|^{p}\phi \, dx+\int_{\Omega}|\nabla\psi_{2}|^{p}\phi \, dx.
    \end{align}
Hence, 
\begin{align*}
   \frac{1} 2 \Phi(w)\leq \frac{1}{2}(\Phi (\psi_{1})+\Phi (\psi_{2}))= \lambda_{1}(\Omega).
\end{align*}
Since $\l_1(\Omega)$ is the infimum, equality must hold throughout. In particular,  equality  holds in \eqref{Dirichlet term convexity}, which implies
$$
\frac{\nabla \psi_{1}(x)}{\psi_{1}(x)}=\frac{\nabla \psi_{2}(x)}{\psi_{2}(x)}\quad \text{a.e. $x\in \Omega$}.
$$
Therefore,  $\psi_{1}(x)=C \psi_{2}(x)$ a.e. in $\Omega$ for some positive constant $C$. This proves simplicity.
\end{proof}
\begin{lemma}\label{lm2.1}
    If $\Omega$ is Lipschitz and $u$ is the first eigenfunction of (\ref{1.2}), then $u\in C(\overline{\Omega})\cap C^{1}(\Omega)$.
\end{lemma}

\begin{proof}
     By Theorem 7.1 and Remark 7.1 in Chapter 4 of \cite{Ladyzhenskaya},  the generalized solution  $u\in W^{1,p}(\Omega,\phi)$ is bounded in any bounded subdomain $\Omega_1\subset\subset \Omega$. Hence $u\in W^{1,p}(\Omega_1)\cap L^{\infty}(\Omega_1)$.  Theorem 1 in \cite{Tolksdorf} then implies $u\in C^{1,\alpha}(\Omega_1)$. The boundary regularity follows from  \cite{DanersTAMS}, which establishes the boundedness of  $u$ in a bounded domain adjacent to a portion of $\partial \Omega$. 
    \end{proof} 
 
\begin{lemma}[Weak comparison principle]\label{comparison lemma}
    Let $\Omega\subset \mathbb{R}^n$ be a  Lipschitz domain. Suppose $u_{1},u_{2}\in W^{1,p}(\Omega,\phi )$ satisfy 
    \begin{align*}
        -\div(\phi |\nabla u_{1}|^{p-2}\nabla u_{1})\leq -\div(\phi |\nabla u_{2}|^{p-2}\nabla u_{2}) \ \ \text{in } \Omega
    \end{align*}
    in weak sense, i.e.,
    \begin{align}\label{weak form of p-lalace}
        \int_{\Omega}|\nabla u_{1}|^{p-2}\nabla u_{1}\nabla v \phi \, dx\leq \int_{\Omega}|\nabla u_{2}|^{p-2}\nabla u_{2}\nabla v \phi \, dx
    \end{align}
    for all $v\geq 0$ and $v\in W_{0}^{1,p}(\Omega,\phi)$.
 If $u_{1}\leq u_{2}$ on $\partial \Omega$, then $u_{1}\leq u_{2}$ in $\Omega$.
\end{lemma}
\begin{proof}
    Choose $v=\max\{u_{1}-u_{2},0\}$. Since $u_{1}\leq u_{2}$ on $\partial \Omega$, we have $v\in W_{0}^{1,p}(\Omega, \phi)$. Inserting this test function into \eqref{weak form of p-lalace} gives
    \begin{align*}
        \int_{\{u_{1}>u_{2}\}}(|\nabla u_{1}|^{p-2}\nabla u_{1}-|\nabla u_{2}|^{p-2}\nabla u_{2})\cdot (\nabla u_{1}-\nabla u_{2})\phi\, dx\leq 0.
    \end{align*}
By the standard monotonicity inequality (see Lemma 1 in \cite{Tolksdorf}), we conclude that $u_{1}\leq u_{2}$ a.e. 
in $\Omega$.
\end{proof}

\begin{lemma}[Hopf's lemma]\label{Hopf's lemma}
    Let $\Omega\subset \mathbb{R}^n$ be a $C^{1,1}$ domain. Let $u\in C^{1}(\overline{\Omega})$ satisfy 
    \begin{align*}
      -\div\left( \phi |\nabla u|^{p-2} \nabla u \right) \ge 0 \qquad \text{in $\Omega$},  
    \end{align*}
    and suppose $u>0$ in $\Omega$ and $u(x_0)=0$ at some $x_0\in \p \Omega$. Then 
    $$
    \frac{\partial u}{\partial \nu}(x_{0})<0,
    $$
where $\nu$ denotes the unit outward vector normal to $\partial \Omega$.
\end{lemma}

\begin{proof}
Let $x_{0}\in \partial \Omega$. There exists an open ball $B_{R}(y)\subset \Omega$ such that $x_{0}\in \partial B_{R}(y)\cap \partial \Omega$.  Define $v(r)=e^{-\alpha r^{2}}-e^{-\alpha R^{2}}$, where $r(x):=|x-y|$ and $\alpha>0$ is a constant chosen later.
A direct computation gives
\begin{align*}
    \nabla v=e^{-\alpha r^{2}}(-2\alpha r\nabla r), \qquad |\nabla v|^{p-2}=e^{-\alpha (p-2)r^{2}}(2\alpha)^{p-2}r^{p-2}|\nabla r|^{p-2}.
\end{align*}
Then
\begin{align*}
    -\div(|\nabla v|^{p-2}\nabla v)
    =&\div(e^{-\alpha (p-1)r^{2}}(2\alpha)^{p-1}r^{p-1}|\nabla r|^{p-2}\nabla r)\nonumber\\
    = &e^{-\alpha (p-1)r^{2}}(2\alpha)^{p-1}r^{p-1}\Delta r+(p-1)e^{-\alpha (p-1)r^{2}}(2\alpha)^{p-1}r^{p-2}\\
    &-(p-1)e^{-\alpha (p-1)r^{2}}(2\alpha)^{p}r^{p}\nonumber\\
    =&(2\alpha)^{p-1}e^{-\alpha (p-1)r^{2}}r^{p-2}\left[n-1+p-1-(p-1)(2\alpha)r^{2}\right],
\end{align*}
and
\begin{align*}
    &-\div(|\nabla v|^{p-2}\nabla v)+|\nabla v|^{p-2}x\cdot \nabla v\\
    =&(2\alpha)^{p-1}e^{-\alpha (p-1)r^{2}}r^{p-2}\left[n+p-2-(p-1)(2\alpha)r^{2}\right]\nonumber\\
    &-e^{-(p-1)\alpha r^{2}}(2\alpha)^{p-1}r^{p-1}x\cdot \nabla r\\
    =&(2\alpha)^{p-1}e^{-\alpha (p-1)r^{2}}r^{p-2}\left[n+p-2-(p-1)(2\alpha)r^{2}-r\nabla r\cdot x\right].
\end{align*}
Choosing $\alpha$ sufficiently large, we obtain 
\begin{align*}
    -\div(|\nabla v|^{p-2}\nabla v)+|\nabla v|^{p-2}x\cdot \nabla v \leq 0  \quad \text{in  $B_{R}(y)\setminus B_{R/2}(y)$},
\end{align*}
which is equivalent to
\begin{align*}
    -\div(\phi |\nabla v|^{p-2}\nabla v)\leq 0 \quad \text{in  $B_{R}(y)\setminus B_{R/2}(y)$}.
\end{align*}
Now choose $\epsilon>0$ small enough so that $w:=\epsilon v$,  on $\partial B_{R/2}(y)\subset \Omega$.
Then $w\leq u$ on $ \partial (B_{R}(y)\setminus B_{R/2}(y))$. By Lemma \ref{comparison lemma}, we have $w \leq u$ in $B_{R}(y)\setminus B_{R/2}(y)$. Consequently,
\begin{align*}
    \frac{\partial u}{\partial \nu}(x_0)\leq \frac{\partial w}{\partial \nu}(x_0)=\epsilon(-2\alpha R e^{-\alpha R^{2}})<0,
\end{align*}
proving the lemma.
\end{proof}
By  \cite[Theorem 8.1]{Serrin}, we also have the strong maximum principle.
\begin{lemma}[Strong maximum principle]\label{strong maximum principle}
    If $u\in C^{1}(\Omega)$ satisfies   in the weak sense
    \begin{align*}
        \begin{cases}
-\div\left(\phi |\nabla u|^{p-2}\nabla u\right) \geq 0, & x \in \Omega, \\[4pt]
u \geq 0, & x \in \Omega,
\end{cases}
    \end{align*}
  then $u(x_{0})=0$ for some $x_{0}\in \Omega$ implies $u(x)\equiv 0$ in $\Omega$.
\end{lemma}
This lemma implies that any nonnegative first eigenfunction corresponding to $\lambda_{1}(\Omega)$  is strictly positive in $\Omega$.

\section{The eigenvalue problem on half-lines}\label{sect3}
In this section, we establish several results for the first eigenvalue and the corresponding eigenfunction on half-lines. Let $I_{\sigma}:=(-\infty, \sigma)$, and consider the following one-dimensional problem:
\begin{equation}\label{1 dimen eigen equation}
\begin{cases}
-\div\left(|w'|^{p-2}w'\phi _1(t)\right)=\lambda_1(\sigma)\phi _1(t)|w|^{p-2}w, \quad t \in I_{\sigma},\\
|w'|^{p-2}w'(\sigma)+\beta|w|^{p-2}w(\sigma)=0,
\end{cases}
\end{equation}
where $\phi _1(t)=(2\pi)^{-1/2}e^{-t^2/2}$.
The  first eigenvalue $\lambda_1(\sigma)$ of \eqref{1 dimen eigen equation} admits the variational characterization
\begin{align}\label{3.1}
\lambda_1(\sigma) = \min_{v \in W^{1,p}(I_\sigma,\phi _1)\setminus\{0\}}
\frac{\displaystyle \int_{-\infty}^{\sigma} (|v'|(t))^p \phi _1(t)\,dt + \beta |v(\sigma)|^p\phi _1(\sigma)}
{\displaystyle \int_{-\infty}^{\sigma} |v(t)|^p \phi _1(t)\,dt}.
\end{align}
Problem \eqref{1 dimen eigen equation} is  closely related to the original problem \eqref{1.2} when $\Omega$ is a half-space. Indeed, for $\sigma\in \R$, define the half-space  
\begin{equation}\label{3.2}
S_{\sigma}:=\{x=(x_1,x_2,\dots,x_n)\in\mathbb{R}^n:x_1<\sigma\}.
\end{equation}
The corresponding $n$-dimensional problem reads
\begin{align}\label{3.3}
\begin{cases}
-\div\bigl(\phi |\nabla u|^{p-2}\nabla u\bigr)=\lambda(S_{\sigma})\,\phi |u|^{p-2}u & \text{in }S_{\sigma},\\[4pt]
|\nabla u|^{p-2}\nabla u\cdot\nu+\beta\,|u|^{p-2}u=0 & \text{on } \{x_1=\sigma\},
\end{cases}
\end{align}
where
\begin{align}\label{3.4}
\lambda_1(S_\sigma)=\inf_{v\in W^{1,p}(S_\sigma,\phi )\setminus\{0\}}
\frac{\int_{S_\sigma}|\nabla v|^p\phi \,dx
+\beta\int_{\{x_1=\sigma\}}|v|^p\phi \,dA}
{\int_{S_\sigma}|v|^p\phi \,dx}.
\end{align}
By Proposition \ref{Theorem1},  the first eigenvalue $\lambda_1(S_\sigma)$ is simple and its corresponding eigenfunctions  are  positive. In this setting, we get the following results: 
\begin{proposition}\label{Theorem2}
Let $u$ be a positive eigenfunction corresponding to $\lambda_1(S_\sigma)$. Then there exists a function $w: I_\sigma \to (0,+\infty)$ such that $u(x)=w(x_1)$ and $w$ solves \eqref{3.1} with $\lambda_1(\sigma) = \lambda_1(S_\sigma)$. Moreover, $w$ is strictly  decreasing on $I_\sigma$.
\end{proposition}

\begin{proof} 
Given the eigenfunction $w$ for the one-dimensional problem \eqref{3.1}, define  $u(x):=w(x_1)$. it is  straightforward to verify that $u$ satisfies \eqref{3.3}. By  Proposition \ref{Theorem1}, the first eigenfunction for the half-space problem is simple and positive; hence  $u$ is the positive eigenfunction corresponding to $\lambda_1(S_\sigma)$. Therefore $\lambda_1(\sigma) = \lambda_1(S_\sigma)$ .

It remains to prove $w$ is strictly decreasing. Fix any $r_0\in(-\infty,\sigma)$. From \eqref{3.3}, we have 
\begin{align*}
\begin{cases}
 -\div\bigl(\phi (x)|\nabla u|^{p-2}\nabla u\bigr)=\lambda_1(\sigma)\phi |u|^{p-2}u>0, & \ x \in S_{r_0},\\[4pt]
u(x)=w(r_0) & x\in  \{x_1=r_0\}.
\end{cases}
\end{align*}
Applying the weak comparison principle (Lemma \ref{comparison lemma}), we obtain
\begin{align*}
u(x)\geq w(r_0)\quad \text{for } x\in S_{r_0}.
\end{align*}
Since $u(x)$ is not  constant, the strong maximum principle (Lemma \ref{strong maximum principle}) implies  that the inequality is strict:
\begin{align*}
   u(x)>w(r_0)\quad \text{for } x\in S_{r_0}. 
\end{align*}
Now we set $v(x)=u(x)-w(r_0)=w(x_1)-w(r_0)$. Then $v$ satisfies 
\begin{align}
\begin{cases}
-\div\bigl(\phi (x)|\nabla v(x)|^{p-2}\nabla  v(x)\bigr)=\l_1(S_{\sigma}) |u(x)|^{p-2}u(x)>0, & \ x \in S_{r_0},\\[4pt]
v(x)=u(x)-w(r_0)=0, &  x\in \{x_1=r_0\}.
\end{cases}
\end{align}
Applying  Hopf's Lemma (Lemma \ref{Hopf's lemma}), we conclude $w'(r_0)<0$. Since $r_0\in (-\infty, \sigma)$ is arbitrary, $w$ is strictly decreasing on $I_\sigma$.
\end{proof}

\begin{remark}\label{remark1}
Once  $w'(r)<0$ is established, standard elliptic regularity theory implies that $w\in C^{\infty}(I_{\sigma})$.
\end{remark}

By Remark \ref{remark1}, problem \eqref{1 dimen eigen equation} can be rewritten in the non-divergence form
\begin{equation}\label{3.5}
\begin{cases}
-(p-1)|w'|^{p-2}w^{''}+t|w'|^{p-2}w'=\lambda_1(\sigma)|w|^{p-2}w,\qquad t\in I_{\sigma}\\
|w'|^{p-2}w'(\sigma)+\beta|w|^{p-2}w(\sigma)=0.
\end{cases}
\end{equation}
We next claim that
\begin{equation}\label{3.6}
\lim\limits_{t\rightarrow -\infty}|w'|^{p-2}w'(t)e^{-\frac{t^2}{2}}=0.
\end{equation}
Indeed, since $w(t)>0$ on $I_{\sigma}$, equation \eqref{1 dimen eigen equation} implies
\begin{equation*}
\left(|w'|^{p-2}w'\phi _1\right)'<0, \quad t\in I_{\sigma}.
\end{equation*}
Thus, the limit 
\begin{align*}
   L:=\lim\limits_{t\rightarrow -\infty}|w'|^{p-2}w'(t)e^{-\frac{t^2}{2}}.
\end{align*}
exists in $(-\infty, +\infty]$. If  $L\neq 0$, then for sufficiently large $|t|$, there exists a constant $C>0$ such that $|w'|\geq Ce^{t^2/(2p-2)}$.  Since  $p>1$, this  would imply 
$$
\int_{-\infty}^\sigma |w'(t)|^pe^{-\frac{t^2}{2}}dt=+\infty,
$$
 contradicting with the fact that $u\in W^{1,p}(\Omega, \phi _1)$. Hence \eqref{3.6} holds. 

Now define
\begin{align}\label{def of beta}
   g(t):=\frac{|w'|^{p-1}}{w^{p-1}},
\end{align}
and we will show  that $g(t)$ is strictly increasing on $I_\sigma$ and satisfies $g(\sigma)=\beta$.
\begin{proposition}\label{monotone of beta prop}
    The function $g(t)$ defined in (\ref{def of beta}) is positive, strictly increasing in $(-\infty,\sigma)$,
and $g(\sigma) = \beta$.
\end{proposition}
\begin{proof} 
From equation \eqref{3.5},
a direct computation gives
\begin{align}\label{3.10}
    g'(t)&=\frac{(p-1)|w'|^{p-3}w'w''}{w^{p-1}}-(p-1)\frac{|w'|^{p-1}}{w^{p}}w'\nonumber\\
    &=\frac{1}{w^{p-1}}(\lambda_1(\sigma) w^{p-1}+t|w'|^{p-1})+(p-1)g^{\frac{p}{p-1}}\nonumber\\
    &=\lambda_1(\sigma) +tg(t)+(p-1)g(t)^{\frac{p}{p-1}}.
\end{align}
and consequently,
\begin{align}\label{g'' equation}
    g''=(t+pg(t)^{\frac{1}{p-1}})g'(t)+g(t).
\end{align}
We first show that $g(t)$ has at most one local minimum point in 
$(-\infty,\sigma]$.   Indeed, if  $t_{0}$ is a critical point of $g$, then \eqref{g'' equation} yields
\begin{align*}
    g''(t_{0})=g(t_{0})>0,
\end{align*}
so $t_{0}$ is necessarily a local minimum. Therefore,  $g$ has no local maximum points in $(-\infty,\sigma]$. If there were  two distinct local minima  $t_1< t_2$, then since $g'(t_i)=0$ and
$g''(t_i)>0$,  there would exist small 
 constants $\delta_1$ and $\delta_2$ small enough, such that $g'>0$ on  $ (t_{1}, t_{1}+\delta_1)$ and $g'<0$ on  $(t_{2}-\delta_2, t_{2})$. This would force the existence of a local  maximum  in between, a contradiction.

 Thus, only three  possibilities remain:
 {\bf Case 1:} $g'(t)>0$ for all $t\in I_\sigma$;
   {\bf Case 2:} $g'(t)\le 0$ for all $t\in I_\sigma$;
   {\bf Case 3:} There exists $t_0\in (-\infty, \sigma)$ such that
 $g'\le 0$ on $(-\infty, t_0)$ and $g'>0$ on $(t_0, \sigma)$. 
We now show that Cases 2 and 3 are impossible. Case 2 may be viewed as the limiting case of Case 3 with $t_0=\sigma$, so we treat them uniformly. Suppose $g'<0$ on $(-\infty, t_0)$, then 
$$
A:=\lim_{t\rightarrow -\infty}g(t)\in (0, +\infty].
$$
If $0<A<+\infty$,  then from \eqref{3.10} we have
$$
g'(t)\to -\infty, \quad \text{as $t \to -\infty$},
$$
which implies $g(t)\to +\infty$, contradicting with $0<A<+\infty$.
Hence $A=+\infty$.  Since
\begin{align*}
    \lambda_1(\sigma)& = \frac{\int_{-\infty}^{\sigma} |w'|^p e^{-\frac{t^2}{2}} \, dt + \beta |w(\sigma)|^pe^{-\frac{\sigma^2}{2}}}{\int_{-\infty}^{\sigma} w^p e^{\frac{-t^2}{2}} \, dt}\\
&= \frac{\int_{-\infty}^{t} |w'|^p e^{-\frac{t^2}{2}} \, dt + \int_{t}^{\sigma} |w'|^p e^{-\frac{t^2}{2}} \, dt + \beta |w(\sigma)|^p e^{-\frac{\sigma^2}{2}} }{\int_{-\infty}^{t} w^p e^{\frac{-t^2}{2}}\, dt + \int_{t}^{\sigma} w^p e^{\frac{-t^2}{2}}\, dt}
\end{align*}
 and for any $t\in I_\sigma$,
\begin{align*}
    \lim_{t\rightarrow-\infty}\frac{ \int_{t}^{\sigma} (w')^p e^{-\frac{t^2}{2}} \, dt + \beta w^p(\sigma) e^{-\frac{\sigma^2}{2}} }{\int_{t}^{\sigma} w^p e^{\frac{-t^2}{2}}\, dt}=\lambda_1(\sigma),
\end{align*}
there exists $t_{3}<\sigma$ such that for any $t<t_{3}$,
\begin{align}
    \frac{ \int_{t}^{\sigma} |w'|^p e^{-\frac{t^2}{2}} \, dt + \beta |w(\sigma)|^p e^{-\frac{\sigma^2}{2}} }{\int_{t}^{\sigma} w^p e^{\frac{-t^2}{2}}\, dt}<\lambda_1(\sigma)+1.
\end{align}
Now choose a positive constant  $K>0$ such that $K>\lambda_1(\sigma)+1$. Since  $A=+\infty$, there exists    $t_{4}$ such that 
\begin{align*}
        g(t)>K^{\frac{p-1}{p}} \quad \text{for any $t\leq t_{4}$},
\end{align*}
 hence
 $|w'(t)|>K^{\frac{1}{p}}w(t)$. Therefore, for $t\le t_4$
\begin{align*}
    \frac{\int_{-\infty}^{t} |w'|^p e^{-\frac{t^2}{2}} \, dt}{\int_{-\infty}^{t} w^p e^{-\frac{t^2}{2}} \,dt} > K.
\end{align*}
Let  $t_5=\min\{t_3,t_4\}$ and define
\begin{align*}
    a_{1}(t):=&\int_{t}^{\sigma} |w'|^p e^{-\frac{t^2}{2}} \, dt + \beta |w(\sigma)|^p e^{-\frac{\sigma^2}{2}},\\
     a_{2}(t):=&\int_{-\infty}^{t} |w'|^p e^{-\frac{t^2}{2}} \, dt,\\
     b_{1}(t):=&\int_{t}^{\sigma} w^p e^{\frac{-t^2}{2}}\, dt,\\
     b_2(t):=&\int_{-\infty}^{t} w^p e^{-\frac{t^2}{2}} \, dt.
\end{align*}
Then for  $t\leq t_5$,
\begin{align*}
   \frac{a_2(t)}{b_2(t)} > K, \quad \ell (t):= \frac{a_1(t)}{b_1(t)}<\lambda_1(\sigma)+1<K. 
\end{align*}
Therefore we conclude that
\begin{align*}
\lambda_1(\sigma)=\frac{a_1 + a_2}{b_1 + b_2} > \frac{K b_2 + \ell b_1}{b_2 + b_1} > \ell = \frac{a_1}{b_1}=\frac{\int_{t}^{\sigma} |w'|^p e^{-\frac{t^2}{2}} \, dt + \beta |w(\sigma)|^p e^{-\frac{\sigma^2}{2}}}{\int_{t}^{\sigma} w^p e^{\frac{-t^2}{2}}\, dt}.
\end{align*}
Choosing a testing function 
\begin{align*}
    \widetilde{w}(t) = 
\begin{cases} 
w(t), & t \in (t_5, \sigma), \\
w(t_5), & t \in (-\infty, t_5)
\end{cases}
\end{align*}
for $\l_1(\sigma)$, we have 
\begin{align*}
    \lambda_1(\sigma)\leq \frac{\int_{t_5}^{\sigma} |w'|^p e^{-\frac{t^2}{2}} \, dt + \beta |w(\sigma)|^p e^{-\frac{\sigma^2}{2}}}{\int_{t_5}^{\sigma} w^p e^{\frac{-t^2}{2}}\, dt+w(t_{5})^p\int_{-\infty}^{t_{5}}e^{-\frac{t^{2}}{2}}dt}< \frac{\int_{t_5}^{\sigma} |w'|^p e^{-\frac{t^2}{2}} \, dt + \beta |w(\sigma)|^p e^{-\frac{\sigma^2}{2}}}{\int_{t_5}^{\sigma} w^p e^{\frac{-t^2}{2}}\, dt}<\lambda_1(\sigma),
\end{align*}
which is a contradiction. This rules out Cases 2 and 3, so Case 1 must hold, i.e.,
 $g'(t)>0$.
\end{proof}
Using the strict increasing property of $g$, we derive the following comparison result.
\begin{proposition}\label{pr3.3}
        Let $r,\sigma \in \mathbb{R}$ with $r\leq \sigma$. Then $\lambda_{1}(r)\geq \lambda_{1}(\sigma)$. 
\end{proposition}
\begin{proof}
    Let $w$ be the eigenfunction corresponding to $\lambda_{1}(\sigma)$.    
    Then
\begin{align*}
\lambda_{1}(\sigma)=\frac{\int_{-\infty}^{\sigma}|w'|^{p}e^{-\frac{t^{2}}{2}}dt+g(\sigma)|w(\sigma)|^{p}e^{-\frac{\sigma^{2}}{2}}}{\int_{-\infty}^{\sigma}|w|^{p}e^{-\frac{t^{2}}{2}}dt}.
\end{align*}
Moreover, $w$ is also the first eigenfunction of one-dimensional $p$-Hermite operator on $I_r$ with Robin parameter $g(r)$, whose first eigenvalue we denoted by $\l_1^{g(r)}(r)$. Thus
    \begin{align*}
         \lambda_{1}^{g(r)}(r)=&\frac{\int_{-\infty}^{r}|w'|^{p}e^{-\frac{t^{2}}{2}}dt+g(r)|w(r)|^{p}e^{-\frac{r^{2}}{2}}}{\int_{-\infty}^{r}|w|^{p}e^{-\frac{t^{2}}{2}}dt}=\l_1(\sigma).
    \end{align*}
Since $g$ is increasing in $r$ and  $\l_1^{\alpha}$ is monotone increasing in  the Robin parameter $\alpha$, we obtain
   \begin{align*}
         \lambda_{1}(\sigma)
         = \lambda_{1}^{g(r)}(r)
         \le \lambda_{1}^{g(\sigma)}(r)
         =\lambda_{1}(r),
    \end{align*}
which proves  the proposition.
\end{proof}

\section{Representation formula for the first eigenvalue}\label{sect4}
Let $\Omega \in \mathcal{G}$, with $\mathcal{G}$ as in Definition \ref{Def 2.1} and let $\psi$ be the first positive eigenfunction
of \eqref{1.2} normalized by $\|\psi\|_{L^p(\Omega,\phi )} = 1$. For each $t > 0$, we introduce the following notation:
\begin{align*}
U_t &:= \{x \in \Omega: \psi(x) > t\}, \\
\partial U_t^{\mathrm{int}} &:= \{x \in \Omega: \psi(x) = t\}, \\
\partial U_t^{\mathrm{ext}} &:= \{x \in \partial \Omega: \psi(x) > t\}.
\end{align*}
For a measurable function $\varphi: \Omega \to [0, \infty)$, we define the functional
\begin{align}\label{4.1}
    \mathcal{F}_\Omega(U_t, \varphi) =: \frac{1}{\gamma(U_t)} \Bigg[
- (p-1)\int_{U_t} \varphi^{p/(p-1)} \phi  \, dx
+ \int_{\partial U_t^{\text{int}}} \varphi \phi  \, dA
+ \beta \int_{\partial U_t^{\text{ext}}} \phi  \, dA
\Bigg],
\end{align}
where $\gamma(U_t):=\int_{U_t}\phi\, dx$ denotes the Gaussian measure of $U_t$.

\begin{proposition}\label{prop4.1}
    Let $\psi > 0$ be the eigenfunction   corresponding to $\lambda_1(\Omega)$. Then
\begin{align*}
\lambda_1(\Omega) = \mathcal{F}_\Omega\left(U_t, \frac{|\nabla \psi|^{p-1}}{\psi^{p-1}}\right)
\end{align*}
for almost every $t \in (0, +\infty)$.
\end{proposition}

\begin{proof}
    We follow the argument in \cite{BucurCV}.
Fix $t \in (0,+\infty)$ and let $\varepsilon \in (0,t)$. Define the function
\begin{align*}
\varphi_\varepsilon := \frac{1}{\psi^{p-1}} \min\left\{ 1, \left( \frac{\psi - t}{\varepsilon} \right)^+ \right\}.
\end{align*}
Then $\varphi_\varepsilon$ 	 converges pointwise as
$\varepsilon \to 0^+$, and
\begin{equation}\label{4.2}
\varphi_\varepsilon \rightarrow \frac{1}{\psi^{p-1}} {\bf 1}_{U_t} \quad \text{as 
$\varepsilon \to 0^+$},
\end{equation}
where ${\bf 1}_{U_t}$ denotes the indicator function of $U_t$. Moreover, $\varphi_\varepsilon \in W^{1,p}(\Omega,\phi )$ and 
\begin{align*}
\nabla \varphi_\varepsilon =
\begin{cases}
-(p-1) \dfrac{\nabla \psi}{|\psi|^p},
& \text{if } \psi > t + \varepsilon \ , \\[1em]
\dfrac{1}{\varepsilon} \left( (p-1)\dfrac{t}{\psi} - p + 2 \right) \dfrac{\nabla \psi}{\psi^{p-1}},
& \text{if } t < \psi < t + \varepsilon, \\[1em]
0, & \text{otherwise}.
\end{cases}
\end{align*}
We now examine each term in the weak formulation  \eqref{weak form} with $u = \psi$ and $v = \varphi_\varepsilon$. For  the gradient term, we obtain
\begin{align*}
&\int_\Omega |\nabla\psi|^{p-2} \nabla\psi \cdot \nabla\varphi_\varepsilon  \phi \, dx\\
=& -(p-1) \int_{U_{t+\varepsilon}} \frac{|\nabla\psi|^p}{\psi^p}\phi  \, dx+ \frac{1}{\varepsilon} \int_{U_t \setminus U_{t+\varepsilon}} \left( (p-1)\frac{t}{\psi} - p + 2 \right) \frac{|\nabla\psi|^p}{\psi^{p-1}} \phi \, dx
\end{align*}
for all $0 < \varepsilon < t$. Applying the coarea formula to the second integral yields
\begin{align*}
&\int_{U_t \setminus U_{t+\varepsilon}} \left( (p-1)\frac{t}{\psi} - p + 2 \right) \frac{|\nabla\psi|^p}{\psi^{p-1}} \phi \, dx\\
=& \int_t^{t+\varepsilon} \left( (p-1)\frac{t}{\tau} - p + 2 \right) \int_{\{\psi=\tau\}} \frac{|\nabla\psi|^{p-1}}{\psi^{p-1}} \phi \,dA \, d\tau
\end{align*}

   Since $\psi \in W^{1,p}(\Omega, \phi )$,   the function
\begin{align*}
s \mapsto \int_t^s \left( (p-1)\frac{t}{\tau} - p + 2 \right) \int_{\{\psi=\tau\}} \frac{|\nabla \psi|^{p-1}}{\psi^{p-1}} \phi\, dA d\tau
\end{align*}
is locally absolutely continuous on $(0,+\infty)$. Using differentiability almost everywhere, we obtain, as $\varepsilon \to 0^+$,
\begin{align*}
\begin{split}
&\frac{1}{\varepsilon} \int_{U_t \setminus U_{t+\varepsilon}} \left( (p-1)\frac{t}{\psi} - p + 2 \right) \frac{|\nabla \psi|^p}{\psi^{p-1}}\phi  dx \\
\to \quad &\left( (p-1)\frac{t}{t} - p + 2 \right) \int_{\{\psi=\tau\}} \frac{|\nabla \psi|^{p-1}}{\psi^{p-1}}\phi dA = \int_{\{\psi=\tau\}} \frac{|\nabla \psi|^{p-1}}{\psi^{p-1}} \phi\, dA
\end{split}
\end{align*}
for almost all $t \in (0,+\infty)$.

For the boundary term, using \eqref{weak form} and  the dominated convergence theorem, we get
\begin{align*}
\int_{\partial \Omega} \beta |\psi|^{p-2} \psi \varphi_\varepsilon \phi dA \to \int_{\partial \Omega \cap U_{t}} \beta \phi dA.
\end{align*}
Similarly,
\begin{align*}
\int_{\Omega} |\psi|^{p-2} \psi \varphi_\varepsilon \phi \, dx \to \int_{U_t}  \phi (x)\, dx = \gamma(U_t).
\end{align*}
Letting $\varepsilon \to 0$ in  the weak identity
\begin{align*}
\int_{\Omega} |\nabla \psi|^{p-2} \nabla \psi \cdot \nabla \varphi_{\varepsilon} \phi\, dx + \int_{\partial \Omega} \beta |\psi|^{p-2} \psi \varphi_{\varepsilon} \phi \,dA = \lambda_1(\Omega) \int_{\Omega} |\psi|^{p-2} \psi \varphi_{\varepsilon} \phi\, dx,
\end{align*}
we obtain
\begin{align*}
-(p-1) \int_{U_t} \frac{|\nabla \psi|^p}{\psi^p}\phi \, dx + \int_{\partial U_t^{\mathrm{int}}} \frac{|\nabla \psi|^{p-1}}{\psi^{p-1}}\phi  \,dA + \int_{\p U_t^{\mathrm{ext}}} \beta \phi \,dA = \lambda_1(\Omega) \gamma(U_t)
\end{align*}
for almost every $t \in (0, +\infty)$.  Rearranging terms yields the desired formula. \end{proof}

\begin{proposition}\label{prop4.2}
Let $\varphi : \Omega \to [0, \infty)$ be a measurable function such that $\varphi \in L^{p'}(U_t,\phi )$ for all $t > 0$, where $p'=p/(p-1)$. Define
\begin{align*}
w := \varphi - \frac{|\nabla \psi|^{p-1}}{\psi^{p-1}},
\qquad
F(t) := \int_{U_t} w \frac{|\nabla \psi|}{\psi} \phi \, dx.
\end{align*}
Then $F : (0,+\infty) \to \mathbb{R}$ is locally absolutely continuous, and
\begin{align}\label{ineqaulity about lambda}
    \mathcal{F}_{\Omega}(U_t, \varphi) \leq \lambda_1(\Omega) - \frac{1}{\gamma (U_t) t^{p-1}} \frac{d}{dt} \left( t^p F(t) \right). 
\end{align}
\end{proposition}

\begin{proof}
From Proposition \ref{prop4.1}, we have
    \begin{align*}
        \mathcal{F}_{\Omega}(U_{t}, \varphi)
= \lambda_{1}(\Omega)
+ \frac{1}{\gamma (U_t)}
\left(
\int_{\partial U_t^{\mathrm{int}}} w \phi \,dA
- (p-1) \int_{U_{t}} (\varphi^{p'} - \frac{|\nabla \psi|^{p}}{\psi^{p}})\phi  \, dx
\right).
    \end{align*}
Applying the elementary inequality
    \begin{align*}
        x^{p'}-v^{p'}\geq p'v^{p'-1}(x-v)
    \end{align*}
    with \( v := |\nabla \psi|^{p-1}/\psi^{p-1} \)  and \(x:=\varphi\), we obtain
\begin{align*}
\varphi^{p'} - \frac{|\nabla \psi|^p}{\psi^p} \geq \frac{p}{p-1} w \frac{|\nabla \psi|}{\psi}. 
\end{align*}
with equality if and only if $w = 0 $. Therefore,
\begin{align*}
\mathcal{F}_{\Omega}(U_t, \varphi) \leq \lambda_1(\Omega) + \frac{1}{\gamma (U_t)}
\left(
\int_{\partial U_t^{\mathrm{int}}} w \phi \,dA
- p \int_{U_t} w \frac{|\nabla \psi|}{\psi} \phi \, dx
\right).
\end{align*}
By coarea formula, we have 
\begin{align*}
    \frac{d}{dt}F(t)&=\frac{d}{dt}\int_{t}^{\infty}\int_{\partial U_s^{\mathrm{int}}}\frac{w}{\psi}\phi \, dA ds=-\frac{1}{t}\int_{\partial U_t^{\mathrm{int}}}w\phi  \, dA.
\end{align*}
Hence 
\begin{align*}
    \int_{\partial U_t^{\mathrm{int}}} w \phi \,dA
- p \int_{U_t} w \frac{|\nabla \psi|}{\psi} \phi \, dx&=-t\frac{dF}{dt}-pF(t)=-\frac{1}{t^{p-1}}\frac{d}{dt}(t^{p}F(t)).
\end{align*}
Substituting this into the previous estimate yields 
\begin{align*}
     \mathcal{F}_{\Omega}(U_t, \varphi) \leq \lambda_1(\Omega) - \frac{1}{\gamma (U_t) t^{p-1}} \frac{d}{dt} \left( t^p F(t) \right),
\end{align*}
proving \eqref{ineqaulity about lambda}.
\end{proof}
We now derive a lower bound for  $\lambda_1(\Omega)$. 
\begin{proposition} \label{theorem3}
   Let $\Omega \in \mathcal{G}$, with $\mathcal{G}$ as in Definition \ref{Def 2.1}, and let $\psi$ be as in Proposition \ref{prop4.1}.
Let $\varphi \in L^{p/(p-1)}(\Omega, \phi )$ be a nonnegative function such that $\varphi \neq |\nabla \psi|^{p-1}/\psi^{p-1}$ on a set of positive measure, and let $\mathcal{F}_\Omega$ be as in \eqref{4.1}. Then there exists a set $T \subset (0, +\infty)$ of positive Lebesgue measure such that for every $t \in T$, 
\begin{align*}
\lambda_1(\Omega) \geq \mathcal{F}_\Omega(U_t, \varphi). 
\end{align*}
\end{proposition}

\begin{proof}
    Suppose, to the contrary, that
    $$\lambda_{1}(\Omega)< \mathcal{F}_{\Omega}(U_{t},\varphi)$$ for almost every $t>0$.
Then by Proposition \ref{prop4.2},
\begin{align}\label{4.4}
    \lambda_{1}(\Omega)<\mathcal{F}_{\Omega}(U_{t},\varphi)\leq \lambda_1(\Omega) - \frac{1}{\gamma (U_t) t^{p-1}} \frac{d}{dt} \left( t^p F(t) \right)
\end{align}
  for almost every $t>0$.  Let $G(t) := t^p F(t)$. From \eqref{4.4}, we have in particular
\begin{align*}
G'(t) = \frac{d}{dt} \left(t^p F(t)\right) \leq 0
\end{align*}
for almost every $t>0$. Hence $G$ is nonincreasing on $(0,+\infty)$.  By definition of $F$ and $w$, and H\"older's inequality, we have
\begin{align*}
F(t) &= \int_{U_t} w \frac{|\nabla \psi|}{\psi} \phi \, dx
= \int_{U_t} \varphi \frac{|\nabla \psi|}{\psi}\phi  \, dx - \int_{U_t} \frac{|\nabla \psi|^p}{\psi^p} \phi \, dx \\
&\leq \int_{U_t} \varphi \frac{|\nabla \psi|}{\psi} \phi \, dx
\leq \frac{1}{t} \int_{U_t} \varphi |\nabla \psi| \phi \, dx
\leq \frac{1}{t} \|\varphi\|_{p'} \|\nabla \psi\|_p
\end{align*}
for all $t>0$. Since $p > 1$, we obtain
\begin{align}\label{Gt at 0}
    \lim_{t \to 0^{+}} G(t) = \lim_{t \to 0^{+}} t^p F(t) \leq \lim_{t \to 0^{+}} t^{p-1} \|\varphi\|_{p'} \|\nabla \psi\|_p = 0.
\end{align}
Since $G(t)$ is nonincreasing, the limit $A:=\lim\limits_{t\rightarrow +\infty}G(t)\in [-\infty, 0]$ exists. From the definition of $G$, 
\begin{align*}
    t^{p}\int_{U_{t}}\varphi \frac{|\nabla\psi|}{\psi}\phi \, dx=G(t)+t^{p}\int_{U_{t}}\frac{|\nabla \psi|^{p}}{\psi^{p}}\phi \, dx.
\end{align*}
Letting $t\rightarrow +\infty$, we note that
\begin{align*}
    \lim_{t\rightarrow +\infty} t^{p}\int_{U_{t}}\varphi \frac{|\nabla\psi|}{\psi}\phi dx&= \lim_{t\rightarrow +\infty}G(t)+ \lim_{t\rightarrow +\infty}t^{p}\int_{U_{t}}\frac{|\nabla \psi|^{p}}{\psi^{p}}\phi dx.
\end{align*}
Since 
\begin{align*}
   0\leq  t^{p}\int_{U_{t}}\frac{|\nabla \psi|^{p}}{\psi^{p}}\phi dx\leq \int_{U_{t}}|\nabla \psi|^{p}\phi dx\rightarrow 0
\end{align*}
as $t\rightarrow +\infty$, we have 
\begin{equation*}
\lim\limits_{t\rightarrow +\infty}t^{p}\int_{U_{t}}\frac{|\nabla \psi|^{p}}{\psi^{p}}\phi dx= 0.
\end{equation*}
Therefore, 
\begin{align}
   0\le  \lim_{t\rightarrow +\infty}t^{p}\int_{U_{t}}\varphi\frac{|\nabla\psi|}{\psi}\phi dx=\lim_{t\rightarrow +\infty} G(t)\le 0.
\end{align}
Thus $\lim\limits_{t\rightarrow +\infty}G(t)=0$.  Since 
$G$ is nonincreasing, we have $G(t)\ge 0$ for every $t>0$. On  the other hand, the preceding estimate yields
$
\lim\limits_{t\to 0^+} G(t)\le 0,
$
hence $
\lim\limits_{t\to 0^+} G(t)= 0.
$
Thus we must have
 $G(t)=0$ for all $t>0$. 

 On the other hand, by assumption $\varphi \neq |\nabla \psi|^{p-1}/\psi^{p-1}$ on a set of positive measure  in $\Omega$. Since the level sets $U_t$ exhaust $\Omega$, there exists some $t_0>0$ such that
 \begin{align*}
     \gamma (U_t\cap \{x\in \Omega: \varphi(x)\neq \frac{|\nabla \psi|^{p-1}}{\psi^{p-1}}\})>0
 \end{align*}
 for all $0<t<t_0$. In this region, the inequality in \eqref{ineqaulity about lambda} is strict, hence the second inequality in \eqref{4.4} is strict, so that $G'(t) < 0$ for almost every $t$ in a 
 neighborhood of $0$, contradicting with $G\equiv 0$.  Hence the desired set 
$T$ of positive measure must exist.

\end{proof}

\section{Proof of Theorem \ref{thm1}}\label{sect5}
 For  Lebesgue measurable set $E\subset \R^n$, its Gaussian measure and Gaussian perimeter are respectively defined by
\begin{align*}
\gamma(E)=\int_{\Omega}\phi (x)\,dx, \quad P_{\phi }(E)=\int_{\partial^* E}\phi (x)\,d\mathcal{H}^{n-1},
\end{align*}
where $\phi (x)$ is the Gaussian density given in  \eqref{1.2}, $\p^* E$ is  reduced boundary of $E$, and  $d\mathcal{H}^{n-1}$ denotes the $( n-1)$-dimensional Hausdorff measure. For sufficiently regular sets, $\partial^* E $ agrees with $\partial E$ up to an  $\mathcal H^{n-1}$-null set, and  $d\mathcal H^{n-1}$ coincides with the usual surface measure $dA$.

The classical Gaussian isoperimetric inequality\cite{Borell,Sudakov74,Ehrhard84} asserts that among all Lebesgue measurable subsets of $\mathbb{R}^n$ with  fixed Gaussian measure, half-spaces minimize the Gaussian perimeter.  Moreover, the minimizer is unique up to rotation about the origin \cite{Carlen01,Cianchi11}. For $a\in\mathbb{R}$, define the half-space
\begin{align*}
S_a:=\{x=(x_1,x_2,\dots,x_n)\in\mathbb{R}^n:x_1<a\}.
\end{align*}
For a measurable set 
$\Omega$, its Gaussian symmetrization $\Omega^{\#}$ is  the half-space $S_{\sigma^{\#}}$,
where $\sigma^{\#}$ is chosen so that
\begin{align*}
\gamma(\Omega)=\gamma(\Omega^{\#})
=\frac{1}{\sqrt{2\pi}}\int_{-\infty}^{\sigma^{\#}}\exp\left(-\frac{t^2}{2}\right)dt.
\end{align*}
The Gaussian isoperimetric inequality can then be stated as follows.
\begin{theorem}[Gaussian isoperimetric inequality]\label{thmiso}
For any Lebesgue measurable set $\Omega\subset\mathbb{R}^n$,
\begin{align*}
P_{\phi }(\Omega)\ge P_{\phi }(\Omega^{\#}),
\end{align*}
where equality holds if and only if $\Omega$   is equivalent to a half-space up to a set of Gaussian measure zero.
\end{theorem}

Let $\psi$ be the first positive eigenfunction of \eqref{1.2} on $\Omega$, and let $\psi_*$ be the positive first eigenfunction for the symmetrized half-space and $\Omega^{\#}=S_{\sigma^{\#}}$. Recalling that $\psi_*$ depends only on the first variable, we write it as
$\psi_*(x_{1},\cdots,x_{n})=\psi_{*}(r)$. We now construct a suitable comparison function  on $\Omega$ by rearranging  the quantity $|\nabla \psi_*|^{p-1}/\psi_*^{p-1}$. Define
\begin{align*}
h(r) := \varphi_*(r) := \frac{|\psi'_*(r)|^{p-1}}{\psi_*(r)^{p-1}}, \quad r \in (-\infty,\sigma^\#].
\end{align*}
For each $t>0$, let $r(t)$ be uniquely determined by the identity 
 $$
 \gamma(U_t)=\gamma(S_{r(t)}), \quad U_t:=\{x\in \Omega: \psi(x)> t
 \}.
 $$
  We then define $\varphi: \Omega\to [0, \infty)$ by setting, for $x\in \Omega$ with  $\psi(x) = t$,
\begin{align*}
\varphi(x) := h(r(t)).
\end{align*}

\begin{lemma}
    The function $\varphi: \Omega\rightarrow \mathbb{R}$ defined above is measurable and  satisfies  $0\leq \varphi(x)\le \beta$ for all $x\in \Omega$.
\end{lemma}
\begin{proof}
Since $\psi$ is continuous on $\overline{\Omega}$ and differentiable in $\Omega$ (by Lemma \ref{lm2.1}), the level sets $U_t$ are measurable, and  $r(t)$ is monotone and hence measurable. Thus $\varphi(x)$ is measurable as a composition of measurable functions. 
By Proposition \ref{monotone of beta prop}, 
$h(r)$
is strictly increasing on $(-\infty, \sigma^\#]$ and satisfies
$h(\sigma^\#)=\beta$. Since $r(t)\le \sigma^\#$ for $t>0$, we obtain
\begin{align*}
    0\le \varphi(x)=h(r(t))\le h(\sigma^\#)=\beta
\end{align*}
for all $x\in \Omega$. The proof is complete.
\end{proof}
As in \cite[Lemma 5.2]{BucurCV}, we recall the following standard inequality.
\begin{lemma}\label{lm5.2}
    There exists an at most countable exceptional set $Q\subset (0,\infty)$ such that 
    \begin{align*}
            P_{\phi }(\partial U_{t})\leq P_{\phi }(\partial U_{t}\cap \Omega)+P_{\phi }(\partial U_{t}\cap \p \Omega)
    \end{align*}
    for all $t\in (0,\infty)\setminus Q$.
\end{lemma}

\begin{proposition}\label{F geq lambda}
    Let $\varphi$ be  defined as above. Then
\begin{align*}
\mathcal{F}_{\Omega}(U_t, \varphi) \geq \mathcal{F}_{S_{\sigma^\#}}(S_{r(t)}, \varphi_*)=\l_1(\Omega^\#) 
\end{align*}
for all $t \in (0,+\infty)\setminus Q$, where $Q$ is the exceptional set from Lemma \ref{lm5.2}. Moreover, if equality holds, then $U_t$ is a half-space up to a set of Gaussian measure zero.
\end{proposition}

\begin{proof}
Since $\varphi(x) = h(r(t))$ on $\partial U_{t}\cap \Omega$ by construction, the isoperimetric inequality (Theorem \ref{thmiso}) and Lemma \ref{lm5.2} yields
\begin{align}\label{5.1}
\int_{\partial S_{r(t)}} \varphi_* \phi dA &= h(r(t)) P_{\phi }(\partial S_{r(t)})\nonumber\\
&\leq h(r(t)) P_{\phi }(\partial U_{t}) \nonumber\\
&\leq h(r(t)) P_{\phi }(\partial U_{t}\cap \Omega) + h(r(t)) P_{\phi }(\partial U_{t}\cap \p \Omega) \nonumber\\
&\leq \int_{\partial  U_{t}\cap \Omega} \varphi \phi\, dA + \int_{\partial  U_{t}\cap \p \Omega} \beta \phi\, dA
\end{align}
for all $t \in (0,\infty)\setminus Q$, where the last inequality follows from $\varphi \leq \beta$. Equality clearly holds if $U_t$ is a half-space and $P_{\phi }(\partial U_{t}\cap \p \Omega)=0$. Moreover, equality in the above implies that  $U_t$ is a half-space up to a set of Gaussian measure zero,  by the equality case in the Gaussian isoperimetric inequality.
 
 
By the definition $r(t)$, we have $\gamma(U_t) = \gamma(S_{r(t)})$ for all $t>0$. Consequently, by \cite{Mazja} (see sec 1.2.3)
\begin{align}\label{5.2}
\int_{U_t} \varphi^{p/(p-1)} \phi \, dx = \int_{S_{r(t)}} \varphi_*^{p/(p-1)} \phi \, dx.
\end{align}
Combining \eqref{5.1}, \eqref{5.2} and the definition of  $\mathcal{F}_\Omega$, we obtain
$$
\mathcal{F}_{\Omega}(U_t, \varphi) \geq \mathcal{F}_{S_{\sigma^\#}}(S_{r(t)}, \varphi_*),
$$
It follows directly from the proof of Proposition  \ref{prop4.1} that 
$$
\mathcal{F}_{S_{\sigma^\#}}(S_{r(t)}, \varphi_*)=\l_1(\Omega^\#)
$$
for each $t>0$. Hence we conclude that
$$
\mathcal{F}_{\Omega}(U_t, \varphi) \geq \mathcal{F}_{S_{\sigma^\#}}(S_{r(t)}, \varphi_*)=\l_1(\Omega^\#), 
$$
which completes the proof.
\end{proof}
We are now ready to prove the main theorem.
\begin{proof}[Proof of Theorem \ref{thm1}]
If $ \phi =\frac{|\nabla\psi|^{p-1}}{\psi^{p-1}}$ a.e.  in $\Omega$, then Proposition \ref{prop4.1}  together with Proposition \ref{F geq lambda} immediately yields the desired inequality. We now assume
 $ \phi \neq \frac{|\nabla\psi|^{p-1}}{\psi^{p-1}} $
 on a set of positive measure in $\Omega$.
By Proposition \ref{theorem3},  there exists a set  $T\subset (0,+\infty)$ of positive Lebesgue measure such that for every $t\in T$,
    \begin{align*}
        \lambda_{1}(\Omega)\geq \mathcal{F}_\Omega(U_t, \varphi).
    \end{align*}
On the other hand, by Proposition \ref{F geq lambda},  for all $t\in T\setminus Q$,
    \begin{align*}
     \mathcal{F}_\Omega(U_t, \varphi)\geq \mathcal{F}_{S_{\sigma^\#}}(S_{r(t)}, \varphi_*) = \l_1(\Omega^\#).
    \end{align*}
This proves the Faber-Krahn inequality.

It remains to discuss the equality case.  Suppose now that equality holds, i.e., $ \lambda_1(\Omega)=\lambda_1(\Omega^\#)$.
By Propositions \ref{prop4.2} and \ref{F geq lambda}, for almost every \(t>0\),
$$ \lambda_1(\Omega) \le F_\Omega(U_t,\phi) \le \lambda_1(\Omega) -\frac{G'(t)}{\gamma(U_t)t^{p-1}}, $$
yields \(G'(t)\le0\). Following the proof of Proposition  \ref{theorem3},  we conclude that  $G\equiv0$. Hence
$\mathcal F_\Omega(U_t,\phi)=\lambda_1(\Omega) $
for almost every \(t>0\). Consequently, equality must hold in the Gaussian isoperimetric inequality used in Proposition \ref{F geq lambda}, and therefore  $U_t$ is a half-space for almost every $t>0$.

Choose a sequence $t_k\downarrow 0$ such that each $U_{t_k}$ is a half-space. Since $U_{t_k}\subset U_{t_{k+1}}$, their boundary hyperplanes are parallel. Indeed, two proper half-spaces with nonparallel boundary hyperplanes cannot be contained one in the other.  Thus, there exist $a_k\in \R$ and $\nu\in\mathbb S^{n-1}$ such that
$$ U_{t_k}=\{x\in \R^n: x\cdot\nu<a_k\}. $$
Since \(\psi>0\) in \(\Omega\), we have
$$ \Omega=\bigcup_k U_{t_k} =\{x\cdot\nu<\sup_k a_k\} $$
up to a set of Gaussian measure zero. Hence \(\Omega\) is a half-space, completing the proof.

\end{proof}
 \bibliographystyle{plain}
	\bibliography{ref}
\end{document}